\documentclass[9pt]{article}
\usepackage[latin1]{inputenc}
\usepackage[small]{titlesec}
\usepackage{amsmath}
\usepackage{amsthm}
\usepackage{amssymb}
\usepackage{standalone}
\usepackage{graphicx, hyperref}
\usepackage{caption}
\usepackage{subcaption}

\theoremstyle{plain}
\newtheorem{thm}{Theorem}[section]

\newtheorem{cor}[thm]{Corollary}
\newtheorem{lem}[thm]{Lemma}

\newtheorem{rem}[thm]{Remark}
\newtheorem{prop}[thm]{Proposition}

\usepackage{color}
\usepackage[small]{titlesec}

\begin{document}

\author{G\'{a}bor Korchm\'{a}ros\thanks{Dipartimento di Scienze di Base e Applicate,
Universit\`{a} degli Studi della Basilicata, Via dell'Ateneo Lucano 10, Potenza, 85100, Italy (Email: gabor.korchmaros@unibas.it)};
 Federico Romaniello\thanks{Dipartimento per l'Innovazione Umanistica, Scientifica e Sociale, Universit\`{a} degli Studi della Basilicata, Via Lanera 20, Matera, 75100, Italy (Email: federico.romaniello@unibas.it)}
and Valentino Smaldore\thanks{Dipartimento di Tecnica e Gestione dei Sistemi Industriali,
Universit\`{a} degli Studi di Padova, Stradella S. Nicola 3, 36100 Vicenza, Italy (Email: valentino.smaldore@unipd.it).}}
\title{Hermitian-Singer Functional and Differential Codes}
\date{}
\maketitle

\begin{abstract}
Algebraic geometry codes on the Hermitian curve have been the subject of several papers, since they happen to have good performances and large automorphism groups. Here, those arising from the Singer cycle of the Hermitian curve are investigated. 

\end{abstract}

\section{Introduction}
In the projective plane $PG(2,q^2)$ of order $q^2$ equipped with homogeneous coordinates $(X_1:X_2:X_0)$ defined over a finite field $\mathbb{F}_{q^2}$, the non-singular plane curve of equation $X_2^qX_0+X_2X_0^q-X_1^{q+1}=0$ is a canonical form of the Hermitian curve ${\mathcal{H}}(q)$. It is well known that a Hermitian curve is a maximal curve over $\mathbb{F}_{q^2}$ with $q^3+1$ points in $PG(2,q^2)$, see \cite{Segre} for more information on the Hermitian curve. Take a divisor $\mathtt {G}$ of ${\mathcal{H}}(q)$ with support $G$ where $G$ is a subset of points of ${\mathcal{H}}(q)$, and define $D$ to be the set of all points of ${\mathcal{H}}(q)$ other than those in $G$. Fix an ordering $(P_1,\ldots,P_N)$ of the points in $D$ where $N=(q^3+1)-|G|$, and define the divisor $\mathtt{D}$ to be the sum $P_1+\ldots+P_N$. Let $\mathcal{L}(\mathtt{G})$ be the Riemann-Roch space associated with $\mathtt{G}$. An algebraic geometric (shortly AG) code on ${\mathcal{H}}(q)$ arises by evaluating the functions in $\mathcal{L}(\mathtt{G})$ on $D$. Such a code $C_{\mathcal{L}}(\mathtt{D},\mathtt{G})$ is the \emph{Hermitian functional code} associated to $\mathtt{G}$.
The dual of the functional code $C_\mathcal{L}(\mathtt{D},\mathtt G)$ is the differential code $C_\Omega(\mathtt{D}, \mathtt G)$, called \emph{Hermitian differential code}. 
%Actually, every differential code is known to coincide with some functional code. % see \cite[Theorem 12.17]{P}.

The functional codes and their duals on the Hermitian curves may have good performances provided that $\mathtt{G}$  is taken appropriately. One of the best known code of this type is the $1$-point Hermitian functional code where $G$ consists of a single point, i.e. $\mathtt{G}=\lambda Q$ where $\lambda$ is a positive integer, and $Q$ is a point $Q\in {\mathcal{H}}(q)$; see \cite{br} and the references therein.  Generalizations of the $1$-point codes have been made so far in two directions. The first consists in replacing $Q$ with a higher degree place of the function field of ${\mathcal{H}}(q)$, which yields differential codes with even better parameters obtained for some values of $\lambda$ such that $q^3-q^2\leq\lambda\leq q^3$; see \cite{CT,KN1,KN2,M}. The second direction is motivated by the fact that the subgroup $\Gamma$ of the automorphism group $PGU(3,q)$ of ${\mathcal{H}}(q)$ which preserves $G$ is also an automorphism group of $C_{\mathcal{L}}(\mathtt{D},\mathtt{G})$ whenever each point of $G$ appears with the same weight in $\mathtt{G}$. In such a context, one may also think of replacing the Hermitian curve by some other curves with many points in $PG(2,q)$, for instance by those investigated in \cite{HK1,HK2,HK3}.   

There has been an ongoing project aiming to determine the Hermitian functional and differential codes $C_{\mathcal{L}}(\mathtt{D},\mathtt{G})$ where $G$ is a (full) orbit of a large subgroup of $PGU(3,q)$ with $q>2$. The $1$-point Hermitian functional code is a particular case, as the $1$-point stabilizer is indeed a maximal subgroup of $PGU(3,q)$. So far, three more cases have been investigated, namely when $G$ is either 
%a chord of ${\mathcal{H}}(q)$ and $|\Gamma|=q^3(q^2-1)$, or 
the intersection of ${\mathcal{H}}(q)$ and an irreducible conic and $\Gamma\cong PGL(2,q)$ with $|\Gamma|=q^3-q$, or it consists of all points of ${\mathcal{H}}(q)$ lying in the subplane $PG(2,q_0)$ where $q=q_0^3$ and $\Gamma\cong PSU(3,q_0)$, see \cite{KNT, KS}. In this paper, we work out the case where $G$ is the orbit of a (cyclic) Singer subgroup $\Gamma$ of $PGU(3,q)$ where $|G|=|\Gamma|=q^2-q+1$ and the normalizer of $\Gamma$ in $PGU(3,q)$ is a maximal subgroup of $PGU(3,q)$ of order $3(q^2-q+1)$. In this case we call $C_{\mathcal{L}}(\mathtt{D},\mathtt{G})$ and $C_{\Omega}(\mathtt{D},\mathtt{G})$  the \emph{Hermitian Singer functional} and the \emph{Hermitian Singer differential code}, respectively. Since the Singer subgroup $\Gamma$ acts on the set $D$ as a sharply transitive group, each $\Gamma$-orbit has the same size $q^2-q+1$, and therefore the Hermitian Singer functional and differential codes are quasi-cyclic.   
Our results are stated in Theorems \ref{main}, \ref{the08112025}, \ref{main1},  and \ref{thm13112025}. 

It seems worth investigating whether the geometry developed in this paper might be adequate to compute the minimum distances of the multiple-Hermitian Singer functional codes $C_{\mathcal{L}}(\mathtt{D},\lambda\mathtt{G})$ where $\lambda\ge 2$. This would in turn provide the minimum distances of their dual codes which are multiple-Hermitian Singer differential codes $C_\Omega(\mathtt{D},\mu\mathtt{G})$ for $\mu\ge 1$. Also, the expansion of the project including functional and differential codes where $G$ is the union of a few $\Gamma$-orbits may produce some more codes with large automorphism groups.

\section{Preliminary Results}\label{sec2}
Notation and terminology are standard. Our references are \cite{Hirschfeld1,HKT,hughes-piper1973,P}.
%In particular, $\mathbb F_q$ stands for a finite field of order $q$, where $q=p^k$, $p$ prime, $AG(2,q)$ for the affine plane over $\mathbb{F}_q$, and $PG(2,q)$ for the projective closure of $PG(2,q)$.
\subsection{Plane curves and their function fields}
For the theory of plane algebraic curves, the reader is referred to \cite[Chapters 1-5]{HKT}. Fix an algebraic closure $\mathbb{K}=\overline{\mathbb{F}}_q$, and let $AG(2, \mathbb{K})$ be the affine plane over $\mathbb{K}$ equipped by coordinates $(X,Y)$. For a non-constant polynomial $F(X,Y)$ over $\mathbb{K}$, the (affine) plane curve $\mathcal F$ of equation $F(X,Y)=0$ is defined to be the set of zeros of $F=F(X,Y)$, i.e.
\[{\bf{v}}(F)=\{(x, y)\in AG(2,\mathbb{K})\mid F(x, y) = 0\}.\]
The \textit{degree} of $\mathcal{F}$ is $\deg F$. A \textit{component} of $\mathcal F$ is any curve $\mathcal{G}=v(G)$ such that $G$ divides $F$. A curve $\mathcal F$ is \textit{irreducible} if $F$ is irreducible; otherwise it is \textit{reducible} and splits into irreducible curves, the components of $\mathcal{F}$. All these definitions are translated from $AG(2,\mathbb{K})$ to its projective closure $PG(2, \mathbb{K})$ equipped with homogeneous coordinates $(X_1:X_2:X_0)$ with $X=X_1/X_0,Y=X_2/X_0$, so that $F$ is replaced by the corresponding homogeneous polynomial $F^*\in \mathbb{K}[X_1,X_2,X_0]$. The projective closure of the affine curve $\mathcal{F}$ is the projective plane curve consisting of all points $P$ whose coordinates satisfy the equation $F^*=0$, i.e.
\[\{P=(x_1:x_2:x_0) \in PG(2,\mathbb{K})\mid F^*(x_1,x_2,x_0)= 0\}.\]
Without loss of generality, $\mathcal{F}$ will also denote the projective closure of $\mathcal{F}$. For a positive integer $n\ge 1$, take as many as $\frac{1}{2}n(n+3)$ points in $PG(2,\mathbb{K})$. By a classical result, there exists some curve of degree $n$ passing through each of those points.  

%Let $\mathcal{F}$ and $\mathcal{G}$ be any two plane curves.  
%For a point $P\in PG(2,\mathbb{K})$, let $I(P,\mathcal{F}\cap \mathcal{G})$ denote the intersection multiplicity between $\mathcal{F}$ and $\mathcal{G}$ at the point $P$. Then the  B\'{e}zout theorem, see \cite[Theorem 3.14]{HKT}, states that either $\mathcal{F}$ and $\mathcal{G}$ have a common component, or \[\deg\mathcal F\cdot \deg\mathcal G=\sum_{P\in\mathcal F\cap\mathcal{G}}I(P, \mathcal F\cap\mathcal G).\]

%Let $\mathbb F_q(\mathcal F)$ be the function field of an irreducible curve $\mathcal F$ with constant field $\mathbb F_q$, regarded as the subfield of the function field $\mathbb{K}(\mathcal F)$ of $\mathcal F$ over $\mathbb{K}$.
From now on we assume that $\mathcal{F}$ is a non-singular plane curve defined over $\mathbb{F}$.
The \emph{divisors} are formal sums of points of $\mathcal F$, and for every nonzero function $h$ in $\mathbb F_q(\mathcal F)$,  ${\rm{div}}(h)$ stands for the \textit{principal divisor} associated with $h$.
For a divisor $\mathtt G$ on $\mathcal F$, the \textit{Riemann-Roch space} $\mathcal L(\mathtt G)$ is the vector space consisting of all functions on $\mathcal{F}$ which are regular outside $\mathtt G$ and have no pole at any point $P$ with multiplicity bigger than  the order of $\mathtt{G}$ at $P$, i.e.  $\mathcal L(\mathtt G)=\{f|{\rm{div}} (f)+\mathtt{G}\succcurlyeq 0$\}.

The dimension $\ell(\mathtt G)$ of $\mathcal L(\mathtt G)$ and  $\deg(\mathtt G)$ are linked
by the \textit{Riemann-Roch Theorem}, see \cite[Theorem 6.70]{HKT}:
\begin{equation}
\label{eq02112025}   \ell(\mathtt G) = \deg(\mathtt G)-\mathfrak  g + 1+\ell(\mathtt W-\mathtt G),
\end{equation}
where $\mathfrak  g$ is the genus of the curve and $\mathtt W$ is a canonical divisor. In particular, for $\deg(\mathtt G)>2\mathfrak  g-2$,
\[\ell(\mathtt G) = \deg(\mathtt G) -\mathfrak  g+1.\]
%Let $\Omega(\mathtt{G})$ be the vector space of rational differential forms $\omega$ on $\mathcal{F}$ with $\rm{div}(\omega)+\mathtt{G}\succcurlyeq 0$, together with the zero form.
Let $\mathcal{U}$ be another non-singular plane curve of equation $U(X,Y)=0$. The intersection divisor $\mathcal{F}\cdot \mathcal{U}$ is given by 
$$\sum_{P\in \mathcal{F}\cap\mathcal{U}} I(P,\mathcal{F}\cap\mathcal{U})P$$ 
where $I(P,\mathcal{F}\cap\mathcal{U})$ is the intersection multiplicity of $\mathcal{F}$ and $\mathcal{U}$ in their common point $P$. The B\'ezout theorem, see \cite[Theorem 3.14]{HKT}, states that 
$$\deg(\mathcal{F})\deg(\mathcal{U})= \sum_{P\in \mathcal{F}\cap\mathcal{U}}I(P,\mathcal{F}\cap \mathcal{U}).$$

The following useful result connecting principal divisors and intersection divisors comes from \cite[Theorem 6.2]{HKT}. Let $\ell_\infty$ denote the line at infinity.  
\begin{lem}
\label{theo6.2HKT} Let $(x,y)$ be a generic point of $\mathcal{F}$. If $\mathcal{U}$ has degree $m$ then 
$$
{\rm{div}}(U(x,y))=\mathcal{F}\cdot \mathcal{U}-m(\ell_{\infty}\cdot \mathcal{F}).
$$
\end{lem}
Now fix a finite subfield $\mathbb{L}$ of $\mathbb{K}$ and assume that the non-singular plane curve $\mathcal{F}$ is defined over $\mathbb{L}$, that is, $\mathcal F$ has equation $F(X,Y)=0$ with $F(X,Y)\in \mathbb{L}[X,Y]$. Take a divisor $\mathtt G=\sum \lambda_i Q_i$ where $Q_1,\ldots,Q_k$ are pairwise distinct points defined over  $\mathbb{L}$, and a canonical divisor $\mathtt{W}$ defined over $\mathbb{L}$. Restrict the functions in the Riemann-Roch space of $\mathtt{G}$ to those defined over $\mathbb{L}$ where a function $f(X,Y)$ is defined over $\mathbb{L}$ if $f(X,Y)=g(X,Y)/h(X,Y)$ with $g(X,Y)$,$h(X,Y)\in \mathbb{L}[X,Y]$. By doing so a vector space $\bf{V}$ over $\mathbb{L}$ arises whose dimension remains $\ell(\mathtt{G})$. Also, the Riemann-Roch theorem (\ref{eq02112025}) holds true for $\bf{V}$. Accordingly, we use the term of Riemann-Roch space of $\mathtt{G}$ over $\mathbb{L}$ for $\bf{V}$, and keep the same notation $\mathcal L(\mathtt G)$. In other words, $\mathcal L(\mathtt G)=\{f|{\rm{div}} (f)+\mathtt{G}\succcurlyeq 0, \mbox{$f$ defined over $\mathbb{L}$}\}$.

In the special case where $\mathbb{L}=\mathbb{F}_{q^2}$, $\mathcal{F}$ is the Hermitian curve of $PG(2,q^2)$, $G=\{Q_1,\ldots,Q_{q^2-q+1}\}$ is a $\Gamma$-orbit of a Singer subgroup of $PGU(3,q)$, and $\mathtt{G}=\lambda(Q_1+\ldots+Q_{q^2-q+1})$ for a positive integer $\lambda$. Then $\mathfrak{g}=\frac{1}{2}(q^2-q)$ and $\deg(\mathtt{G})=\lambda(q^2-q+1)$. Since $|G|>2\mathfrak{g}-2$, the Riemann-Roch theorem yields in this case
$$\ell(\mathtt G) =\textstyle{\frac{1}{2}}\big((2\lambda-1)(q^2-q)+2\lambda\big).$$

%Moreover, there 

%Let $\mathtt D =Q_1+\ldots+Q_n$ be another divisor where $Q_1,\ldots, Q_n$ are distinct points of $\mathcal{F}$ in $PG(2,q)$, such that $Supp(\mathtt G)$ is disjoint from $Supp(\mathtt D)$.
%The Hermitian curve ${\mathcal{H}}(q)$ is defined to be the non-singular plane curve in $PG(2,q^2)$ of affine equation
%$F(X,Y)=Y^q+Y-X^{q+1}=0$ and viewed as a curve in  %$\mathbb{F}$.
%The automorphism group of ${\mathcal{H}}(q)$ is $PGU(3,q)$ and it acts on the set ${\mathcal{H}}(q)(\mathbb{F}_{q^2})$ of all its points in $PG(2,q^2)$. Moreover, ${\mathcal{H}}(q)(\mathbb{F}_{q^2})$
%has size $q^3+1$ and it is the set of all isotropic points of a (non-degenerate) unitary polarity. In $PG(2,q^6)$, let $\mathcal{F}$ be the plane curve of homogeneous equation $X_0X_1^q+X_1X_2^q+X_2X_0^q=0$.
\subsection{Projective unitary group and its Singer subgroups}\label{singergroup}
%For group theory, our reference is \cite{cameron}.
The projective unitary group $PGU(3,q)$ is a subgroup of the projective group $PG(3,q^2)$ of the projective plane $PG(2,q^2)$ defined over the finite field $\mathbb{F}_{q^2}$ of order $q^2$. More precisely, $PGU(3,q)$ is the subgroup of $PGL(3,q^2)$ which leaves the set of all isotropic points of (non-degenerate) unitary polarity, equivalently the set of all points of a Hermitian curve in $PG(2,q^2)$.
The structure and the action of the subgroups of $PGU(3,q)$ are well understood, see \cite{har,hoffer1972,oli, mi}, and  \cite[Theorem A.10]{HKT}.

In this paper, we focus on the Singer subgroup $\Gamma$ of $PGU(3,q)$ with $q>2$. As $\Gamma$ is a subgroup of the Singer group of $PG(2,q^2)$, no non-trivial element in $\Gamma$ fixes a point in $PG(2,q^2)$. Therefore, each orbit of $\Gamma$ has maximum length equal to $q^2-q+1$, and $q+1$ of such $\Gamma$-orbits provide a partition of the pointset of the Hermitian curve left invariant by $PGU(3,q)$. Let $\omega$ be a generator of $\Gamma$, prior to a suitable change of the projective frame of $PG(2,q^2)$, the matrix representing $\omega$ is of the form
\[M=\left(\begin{array}{ccc}
1 & 0 & 0\\
0 & 1& 0\\
a &b& c
\end{array}\right),\]
where $X^3+aX^2+bX+c\in\mathbb F_{q^2}[X]$ is an irreducible polynomial. Unfortunately, although the equation of the Hermitian curve in the new coordinates, as well as, the powers of $M$ and the $\Gamma$-orbit may be computed with standard methods from linear algebra, the resulting equations and relations are so complicated that their interpretations may be hard; see Section \ref{sec6}. To avoid this inconvenience, we use a previous idea; see \cite{CK1998,CK1997,CKT}. The matrix $M$ has three distinct eigenvalues, each defined over a cubic extension $\mathbb{F}_{q^6}$ of $\mathbb{F}_{q^2}$. Therefore, to work with a diagonal matrix representation of $\omega$, it is useful to view $PGU(3,q)$ as a subgroup of $PGL(3,q^3)$.
%As the matrix representation of the Singer cycle is hard to work with, we will embed our geometry in another frame, while the projective plane $PG(2,q^2)$ is embedded into $PG(2,q^6)$.
%Following the steps of the papers
%\cite{CK1997,CK1998,CKT} can be used in the following setting: %where a subplane $\Pi\cong PG(2,q^2)$ of $PG(2,q^6)$ is taken to a suitable non-canonical position. More precisely,
To this scope, fix a projective reference system in $PG(2,q^6)$ with homogeneous coordinates $(X_1:X_2:X_0)$. For a primitive $(q^2+q+1)$-root of unity in $\mathbb{F}_{q^6}$, let $\Pi$ be the subgeometry of $PG(2,q^2)$ with point-set
\[\mathcal{P}=\{P(a^i:a^{i(q^2+1)}:1)\mid i=0,1,\ldots,q^4+q^2\},\] and  line-set
{{\small}\[\Lambda=\{\ell: a^iX_1+a^{i(q^2+1)}X_2+X_0=0 \mid i=0,1,\ldots,q^4+q^2\}.\]}
This subgeometry $\Pi$ is a projective plane of order $q^2$, projectively equivalent to $PG(2,q^2)$ in $PG(2,q^6)$; see \cite{CK1997,CK1998,CKT}. In other words, $\Pi$ is a projective subplane of $PG(2,q^6)$ lying in a non-canonical position whose lines are the sections of the pointset of $\Pi$  by the lines of $PG(2,q^6)$ with equations of the form $tX^1+ t^{q^2}X_2+X_3=0$, as $t$ runs over the $(q^4+q^2+1)$-th roots of unity. The Frobenius collineation $\Phi_{\Pi}$ of $\Pi$, that is, the (non-linear) collineation of $PG(2,q^6)$ fixing $\Pi$ pointwise, is the product $\rho\circ \Phi$ where $\rho:(X_1:X_2:X_0)\mapsto (X_0:X_1:X_2)$ and $\Phi:(X_1:X_2:X_0)\mapsto (X_1^{q^2}:X_2^{q^2}:X_0^{q^2})$.     
A generator $\sigma$ of a Singer group of $\Pi$ is represented by the matrix
\[A=\left(\begin{array}{ccc}
\alpha   & 0& 0\\
0& \alpha^{q^2+1}&0\\
0 & 0&1
\end{array}\right),\]
where $\alpha$ is a primitive $(q^4+q^2+1)$-th root of the unity of $\mathbb{F}_{q^6}$.
\section{Families of Hermitian curves}
From now on, we assume that $q>2$.

For any $(q^2+q+1)$-th root $t$ of unity, let $H_t$ be the curve of $PG(2,q^6)$ with homogeneous equation $tX_1X_2^q+t^{q^2+1}X_2X_0^q+X_0X_1^q=0$.  A straightforward computation shows that $\Gamma$ preserves $H_t$.
Also, $t^{q^3-1}=1$ and hence $t^{q^3+q}=t^{q+1}$. 

\begin{prop}
\label{pro30102025} $H_t$ is a Hermitian curve of $\Pi$.
\end{prop}
\begin{proof} A straightforward computation shows that $H_t$ has non-singular point. Therefore, $H_t$ and ${\mathcal{H}}(q)$ have the same genus $\mathfrak{g}$ equal to $\frac{1}{2}q(q-1)$. We show that $H_t$ has as many as $q^3+1$ points in $\Pi$. A point $P=(\gamma:\gamma^{q^2+1}:1)$ of $\Pi$ belongs to $H_t$ if and only if
\begin{equation}
\label{eq30102025}
1+t\gamma^{q^3+1}+t^{q^2+1}\gamma^{q^2-q+1}=0,\,\, \gamma^{q^4+q^2+1}=1.
\end{equation}
Equivalently,
$$1+t\gamma^{(q^2-q+1)(q+1)}+t^{q^2+1}\gamma^{q^2-q+1}=0,\,\,\gamma^{(q^2-q+1)(q^2+q+1)}=1.$$
Let $\delta=\gamma^{q^2-q+1}$. Then $1+t\delta^{q+1}+t^{q^2+1}\delta=0$ with $\delta^{q^2+q+1}=1.$
The equation $tX^{q+1}+t^{q^2+1}X+1=0$ has as many as $q+1$ pairwise distinct roots $\varepsilon$ in an algebraic closure of $\mathbb{F}_{q^2}$. As $t\varepsilon^{q+1}+t^{q^2+1}\varepsilon+1=0$ yields $t^q\varepsilon^{q^2+q}+t^{q^3+q}\varepsilon^q+1=0$, we have $t^q\varepsilon^{q^2+q+1}+t^{q+1}\varepsilon^{q+1}+\varepsilon=0$ whence $t^q(\varepsilon^{q^2+q+1}+t\varepsilon^{q+1}+t^{q^2+1}\varepsilon)=0$ follows. Hence 
$\varepsilon^{q^2+q+1}=1$. Therefore, the elements $\gamma$ satisfying (\ref{eq30102025}) are determined in the following way. First, the $q+1$ solutions $\varepsilon$ of Equation $tX^{q+1}+t^{q^2+1}X+1=0$ are computed, all belong to $\mathbb{F}_{q^3}$. After that for each such $\varepsilon$ the $q^2-q+1$ solutions of Equation $X^{q^2-q+1}=\varepsilon$ are computed. All together, $q^3+1$ solutions $\gamma$ are obtained. Therefore, $H_t$ has exactly $q^3+1$ points in $\Pi$. Moreover, $H_t$ is defined over $\Pi$ since the Frobenius map $\Phi_{\Pi}$ preserves $H_t$. In fact, since $\Phi_{\Pi}(x_1)=x_0^q,\Phi_{\Pi}(x_2)=x_1^q, \Phi_{\Pi}(x_0)=x_2^q$ and  
$t^{q^4+q^2}=t^{-1}$, we have
\begin{align*}
&t\big(tX_1X_2^q+t^{q^2+1}X_2X_0^q+X_0X_1^q\big)^{q^2}=\\
&t^{q^2+1}X_1^{q^2}X_2^{q^3}+X_2^{q^2}X_0^{q^3}+tX_0^{q^2}X_1^{q^3}=\\
&t\Phi_{\Pi}(X_1)\Phi_{\Pi}(X_2)^q+t^{q^2+1}\Phi_{\Pi}(X_1)\Phi_{\Pi}(X_2)^q+\Phi_{\Pi}(X_0)\Phi_{\Pi}(X_1)^q.
\end{align*}
%Let $\sigma$ be the Frobenius map of $PG(2,q^6)$ which fixes $\Pi$ pointwise. If $H_t$ is not defined over $\Pi$, then $\sigma(H)\ne H$, and hence $\sigma(H)\cap H\le (q+1)^2$ by the B\'ezout's theorem. On the other hand, $\sigma(H_t)\cap H_t\ge q^3+1$ as each point of $H_t$ lying on $\Pi$ is also a point of $\sigma(H_t)$. Therefore, $\sigma(H_t)=H_t$ and hence $H_t$ is a curve of $\Pi$. 
Furthermore, $H_t$ has genus $\frac{1}{2}q(q-1)$ and it has as many as $q^3+1$ points in $\Pi$. From \cite{hirschfeld-storme-thas-voloch1991}, $H_t$ is the Hermitian curve of $\Pi$.
\end{proof}
\begin{rem} {\em{A proof for Proposition \ref{pro30102025} can also be obtained from \cite[Proposition 2.5]{CEK} by interpreting the classical unital as the set of points of a Hermitian curve.}}  
\end{rem} Let $A_1=(1:0:0),A_2=(0:1:0),A_0=(0:0:1)$ be the vertices of the fundamental triangle of $PG(2,q^6)$. Let $\ell_{i,j}$ be the line through $A_i$ and $A_j$.   
\begin{lem}
\label{lem13112025} $I(A_i,H_t\cap \ell_i)=q$ and $I(A_i,H_t\cap \ell_{i+1})=1$.  
\end{lem}
\begin{proof} The Frobenius collineation $\Phi_{\Pi}$ takes $A_i$ to $A_{i+1}$ (with $A_3=A_0$). Therefore, the line $\ell_{i,i+1}$ is tangent to $H_t$ at the point $A_i$. 
From the Fundamental Equation, see \cite[Section 10.8]{HKT}, the intersection divisor $H_t\cdot \ell_{i,i+1}=qA_i+A_{i+1}$ whence the claims follow. 
\end{proof}
\begin{prop}
\label{pro02112025A} The $\Gamma$-orbit through any point of $\Pi$ is shared by exactly $q+1$ Hermitian curves $H_t$.  
\end{prop}
\begin{proof} We use the computation carried out in the proof of Proposition \ref{pro30102025}. 
%We have already shown that $H_t$ is a Hermitian curve. 
For a $q^4+q^2+1$-th root $\gamma$ of unity, the point $P=(\gamma:\gamma^{q^2+1}:1)$ of $\Pi$ belongs to $H_t$ if and only
if (\ref{eq30102025}) holds. Let $\delta=\gamma^{q^2-q+1}$. Since $t^{q^3+q}=t^{q+1}$, then $H_t$ passes through $P$ if and only if $1+t^q\delta^{q^2+q}+t^{q+1}\delta^q=0$. By $\delta^{q^2+q+1}=1$, the latter equation is equivalent to  
$\delta+t^q+t^{q+1}\delta^{q+1}=0$. Replacing $t$ by $t^{-1}$ gives $\delta t^{q+1}+t+\delta^{q+1}=0$. Therefore, it is enough to show that each root of the polynomial $P(X)=\delta X^{q+1}+X+\delta^{q+1}$ in $\overline{\mathbb{F}}$ is a ($q^2+q+1$)-th root of unity. This claims follows from $\delta P(X)^qX=X^{q^2+q+1}\delta^{q+1}+\delta X^{q+1}+\delta^{q^2+q+1}X$ and $\delta^{q^2+q+1}=1$.       
%Therefore, $(t^q,t^{q+1})$ is a solution of the linear system of equations
%\begin{align*} 
%&\delta^{q^2+q}X+\delta^qY+1=0;\\
%&X+\delta^{q+1}Y+\delta=0. 
%\end{align*}
%The determinant of this system equals $\delta^q-1$ which
%From the proof of Proposition \ref{pro30102025}, among the $q^2+q+1$-th root of unities, exactly $q+1$ are also roots of the polynomial $X^{q+1}+X+1=0$, and hence of the polynomial $X^{q^2+1}+X+1=0$.
\end{proof}
\begin{prop}
\label{pro02112025B} Let $H_\tau$ be a Hermitian curve of $\Pi$ passing through $G$. For every point $P\in H_\tau$ lying in $\Pi$, there exists a Hermitian curve $H_t$ of $\Pi$ through $P$. 
\end{prop}
\begin{proof} The points of $H_\tau$ lying in $\Pi$ is partitioned into $\Gamma$-orbits, one of them passes through $P$. From Proposition \ref{pro02112025A}, $P$ also lies on some $H_t$ other than $H_\tau$.  
\end{proof}
\begin{prop}
\label{pro13112025} Two distinct Hermitian curves $H_t$ and $H_u$ have at most $q^2-q+1$ common points in $\Pi$. For each such point $P\in \Pi$, $I(P,H_t\cap H_u)=1$.  
\end{prop}
\begin{proof} For any common point $P\in \Pi$, the whole $\Gamma$-orbit of $P$ is in $H_t\cap H_u$. Therefore, if $H_t$ and $H_u$ had more than $q^2-q+1$ common points in $\Pi$, (or a point $P$ with $I(P,H_t\cap H_u)\ge 2$), then they would have at least $2(q^2-q+1)$ such points (or at least $2(q^2-q+1)$ intersection multiplicities), contradicting the B\'ezout theorem for $q>3$ by $2(q^2-q+1)>(q+1)^2$. Actually this contradiction holds true for $q=3$ as $H_t\cap H_u$ contains three more points, namely the fundamental points of our projective reference system in $PG(2,q^6)$.   
\end{proof} 
Propositions \ref{pro02112025A} \ref{pro02112025B} and \ref{pro13112025} have the following corollary.
\begin{cor}
\label{pro13112025A}  $H_u\cap H_t$ contains a point lying in $\Pi$. 
\end{cor}
\begin{lem}
\label{lem13112025B} For $t\neq u$, $I(A_i,H_t\cap H_u)=q(A_1+A_2+A_0)$, and $$H_t\cdot H_u=\sum_{P\in \Omega}P+q(A_1+A_2+A_0)$$ where $\Omega$ is a $\Gamma$-orbit in $\Pi$.  
\end{lem}
\begin{proof} From Lemma \ref{lem13112025}, $I(A_i,H_t\cap H_u)\ge q$. From Corollary \ref{pro13112025A}, there exists a common point $P\in \Pi$ of $H_t$ and $H_u$. Since the $\Gamma$-orbit of $P$ is in $H_t\cap H_u$,
the B\'ezout theorem yields $(q+1)^2\ge 3q +q^2-q+1$ whence $I(A_i,H_t\cap H_u)=q$ and $H_t\cap H_u$ consists of the vertices of the fundamental triangle together with the points of a unique $\Gamma$-orbit. 
\end{proof}

The advantage of representing the Hermitian curve of $\Pi$ by $H$ is that the generator $\omega$ of the Singer subgroup $\Gamma$ is diagonal, as being represented by the matrix $A^{q^2+q+1}$ equal to
\[B=\left(\begin{array}{ccc}
\beta   & 0& 0\\
0& \beta^{q^2+1}&0\\
0 & 0&1
\end{array}\right)
=\left(\begin{array}{ccc}
\beta   & 0& 0\\
0& \beta^{q} &0\\
0 & 0&1
\end{array}\right)
\]
where $\beta$ is a primitive $(q^2-q+1)$-th root of the unity of $\mathbb{F}_{q^6}$. Thus the $\Gamma$-orbit of a point $P=(\alpha:\alpha^{q^2+1}:1)$ is the set
$$K=\{(\alpha\beta^i:(\alpha\beta^i)^{q^2+1}:1)|i=0,\ldots,q^2-q\}.$$

The set $K$ is well known in Finite geometry under the name of a Singer-arc. It is a complete $(q^2-q+1)$-arc of $\Pi$, the largest known complete arc other than the irreducible conics. Here we need a new property of $K$ that we state in the following proposition. 
\begin{prop}
\label{pro08112025} There exists a point in $\Pi\setminus H_t$ such that the number of chords of $H_t$ through $U$ which are disjoint from $K$ is at least $\frac{1}{2}(q-1)^2-1$. 
\end{prop}
\begin{proof} For a point $P$ of $\Pi$ not lying in $H_t$, let $c_P$ denote the number of chords of $H_t$ passing through $P$. Let $c$ be the maximum value of $c_P$ when $P$ ranging in $\Pi\setminus H_t$.
Counting the incident point-line pairs $(P,r)$, where $P$ runs over the $q^4-q^3+q^2$ points in $\Pi\setminus H_t$ while $r$ does over $\frac{1}{2}(q^2-q+1)(q^2-q)$ the chords of $K$, gives 
$$(q^4-q^3+q^2)c\ge \sum_P c_P=\textstyle{\frac{1}{2}}(q^2-q+1)(q^2-q)(q^2-q)$$ 
whence $c\ge \frac{1}{2}(q-1)^2$ follows. Take a point $U$ for which $c_U=c$. As $K$ is an arc, through $U$ we find exactly $(q^2-q+1)-2c$ $1$-secants to $K$. Since $U$ is the common point of exactly $q^2-q$ chords of $H_t$, it turns out that at least $q^2-q-(q^2-q+1-c)$ chords of $H_t$ through $U$ are disjoint from $K$. Since $c\ge \frac{1}{2}(q-1)^2$, the claim follows.
\end{proof}

%\begin{prop}
%\label{pro03112025} 
%\end{prop}

%the subplane $PG(2,q^2)$ of $PG(2,q^6)$ is mapped onto another (non-canonical) subplane of order $q^2$.

\section{Functional and differential codes from the Singer group} 
\subsection{Hermitian Singer functional code}\label{sec3.1} Assume that the non-singular plane curve $\mathcal{F}$ is defined over a finite subfield $\mathbb{L}$ of $\mathbb{K}$. Let $\mathtt{G}$ be a divisor of $\mathcal{F}$ over $\mathbb{L}$ with support $G$. Let $D$ be the set of all points of $\mathcal{F}$ in $PG(2,\mathbb{L})$ other than those in $G$. Fix an ordering $(Q_1,Q_2,\ldots,Q_n)$ of the points in $D$, and let $\mathtt{D}=Q_1+\ldots+Q_n$, the associated divisor. 
Let $\mathcal{L}(\mathtt{G})$ be the Riemann Roch space of $\mathtt{G}$ defined over $\mathbb{L}$. 

For any function $f\in\mathcal L(\mathtt G)$ defined over $\mathbb{L}$, the evaluation of $f$ on $\mathtt{D}$ is given by $ev_{\mathtt D}(f) =(f(Q_1),\ldots, f(Q_n))$. Assume that $n>\deg(G)>2\mathfrak{g}-2$. Then the arising evaluation map $ev_{\mathtt D}:\mathcal L(\mathtt G)\rightarrow\mathbb{L}$ is $\mathbb{L}$-linear and injective. Its image is the {\emph{functional}}-code $C_{\mathcal{L}}(\mathtt{D},\mathtt{G})$ of length $n$, dimension $k=\deg(\mathtt{G})-\mathfrak{g}+1$ and minimum distance $d\geq \delta$ where $\delta=n-\deg(\mathtt{G})$ is the \emph{designed minimum distance}.

%Since we are evaluating over the points of the divisor $\mathtt D$, the length of the code is $q^3-q^2+q=q(q^2-q+1)$, given by the point on the curve $\mathcal D$ not on the divisor $\mathtt G$.
From now on, $\mathbb{L}$ stands for the finite subfield of $\mathbb{K}$ of order $q^2$. Fix a $(q^2+q+1)$-th root of unity $\tau$ in $\mathbb{F}_{q^6}$, and let $H_\tau$ be a Hermitian curve of $\Pi$ of homogeneous equation $\tau X_1X_2^q+\tau^{q^2+1}X_2X_0^q+X_0X_1^q=0$. For $G$, take the orbit of a point $P\in H_\tau\cap \Pi$ under the action of a Singer subgroup $\Gamma$ of $PGU(3,q)$. Let $C_{\mathcal{L}}(\mathtt D,\mathtt G)$ be the functional code on $H_\tau$ where $\mathtt{G}=P_1+\ldots P_N$ with $G=\{P_1,\ldots P_N\}, N=q^2-q+1$, and $\mathtt{D}=Q_1+\ldots+Q_n$ with $D=\{Q_1,\ldots,Q_n\},n=q^3+1-(q^2-q+1)$ such that $G\dot\cup D$ is a partition of the set of all points of $H_\tau$ on $\Pi$.  
We call $C_\mathcal{L}(\mathtt{D},\mathtt{G})$ the \emph{Hermitian Singer functional code}. For every positive integer $\lambda\ge 2$, a generalization is obtained whenever $\mathtt{G}$ is replaced by $\lambda\mathtt{G}$. 

Since $\deg(\mathtt G)=q^2-q+1$ is greater than $2\mathfrak g-2=q^2-q-2$, the Riemann-Roch space $\mathcal{L}(\mathtt G)$ has dimension $\deg(\mathtt G)-\mathfrak g+1=q^2-q+1-\frac{1}{2}(q^2-q)+1=\frac{1}{2}(q^2-q)+2$. We determine  a basis of $\mathcal L(\mathtt G)$. From Proposition \ref{pro02112025A}, there exists another Hermitian curve $\mathcal{H}_t$ of affine equation $F(X,Y)=t XY^q+t^{q^2+1}Y+X^q=0$ through $G$. Look at the intersection divisor 
    $H_\tau\circ H_t$ of $H_\tau$ and $H_t$. 
%If $V$ is one of the three vertices of the fundamental triangle 
%    in $PG(2,q^6)$, then $I(V,H_\tau\cap H_t)=q$. Moreover, $I(P_i,H_\tau\cap H_t)\ge 1$. Thus, from the B\'ezout
%theorem, $I(P_i,H_\tau\cap H_t)=1$, and the intersection of $H_\tau\cap H$ contains no point other than those in $G$ and the vertices $V_1,V_2,V_3$ of the fundamental triangle.
From the second claim of Lemma \ref{lem13112025B}, 
$F(X,Y)^{-1}\in \mathcal{L}(\mathtt{G})$. To find more functions in $\mathcal{L}(\mathtt{G})$, consider the sides $\ell_1,\ell_2,\ell_3$ of the fundamental triangle of $PG(2,q^6)$. From Lemma \ref{lem13112025}, for each such line $\ell$, $I(V,H_t\cap \ell)=q$. Therefore, the intersection divisor $H_t\circ (\ell_1\ell_2\ell_3)=(q+1)(V_1+V_2+V_3)$. It turns out that for any polynomial $G(X,Y)$ of degree at most $q-2$ which represents a curve defined over $\Pi$, 
\begin{equation}
\label{eq04112025}
    \frac{G(X,Y)\ell_1\ell_2\ell_3}{F(X,Y)}\in \mathcal{L}(\mathtt{G}).
\end{equation}
The linear system of such polynomials $G(X,Y)$ has dimension $\frac{1}{2}(q^2-q)$. Therefore, the functions in (\ref{eq04112025}) form a subspace $\Lambda$ of $\mathcal{L}(\mathtt{G})$ with $\dim(\Lambda)=\frac{1}{2}(q^2-q)$. 
Now, take point $Q\in D$. By Proposition \ref{pro02112025A}, there is a Hermitian curve $H_u$ through $Q$. Since 
$\Pi\cap (H_\tau\cap H_t)=G$ and $Q\not\in G$, $H_u$ is neither $H_\tau$, nor $H_t$. Let $H(X,Y)=uXY^q+u^{q^2+1}Y+X^q=0$ 
be an affine equation of $H_u$. Then $H(X,Y)/F(X,Y)\in \mathcal{L}(\mathtt{G})$ and $H(X,Y)/F(X,Y)\not\in \Lambda$. 
Thus $\Lambda$ together with $H(X,Y)/F(X,Y)$ and a non-zero constant function, determine a basis of 
$\mathcal{L}(\mathtt{G})$. Since $H(X,Y)/F(X,Y)$ has exactly $q^2-q+1$ zeros in $D$, the weight of $H(X,Y)/F(X,Y)$ in $C_{\mathcal{L}}(\mathtt D,\mathtt G)$ equals $q^3+1-2(q^2-q+1)$. On the other hand, the designed minimum distance of $C_{\mathcal{L}}(\mathtt D,\mathtt G)$ is $n-\deg(\mathtt{G})=q^3+1-2(q^2-q+1)$. Therefore, the following result is obtained. 
\begin{thm}\label{main} The functional Hermitian Singer code 
    $C_{\mathcal{L}}(\mathtt D,\mathtt G)$ is a $$[q(q^2-q+1), \textstyle{\frac{1}{2}}(q^2-q)+2, q^3+1-2(q^2-q+1)]_{q^2}
    $$ linear code whose  minimum distance is equal to the designed minimum distance. 
\end{thm}
The subspace $\Lambda$ in the above discussion is a subcode of $C_{\mathcal{L}}(\mathtt{D,\mathtt{G}})$ of length $q(q^2-q+1)$, dimension $\frac{1}{2}(q^2-q)$ and designed minimum distance $q^3-2(q^2-q-1)$. We show that the real minimum distance equals $q^3-2(q^2-q)+1$. The zeros of the functions in $\Lambda$ are contained in the intersection of $H_t$ and a curve of equation $G(X,Y)$ of degree at most $q-2$. Since $H_t$ is absolutely irreducible, the B\'ezout theorem yields that the size of this intersection at most $(q+1)(q-2)$. Actually, this bound is attained for some functions in $\Lambda$. For instance, as $\frac{1}{2}(q-1)^2-1\ge q-2$,  Proposition \ref{pro08112025} gives a method to find such a function by taking $q-2$ chords through a suitable chosen point $U\in \Pi\setminus H_t$, and define $G(X,Y)$ as the product of the linear polynomials representing those $q-2$ lines. Therefore, the following result holds.  
\begin{thm} The Hermitian Singer functional code 
\label{the08112025} $C_{\mathcal{L}}(\mathtt{D},\mathtt{G})$ contains a subcode $$[q(q^2-q+1), \textstyle{\frac{1}{2}}(q^2-q), q^3-2(q^2-q-1)]_{q^2}$$ whose minimum distance improves on the designed minimum distance by $3$.
\end{thm} 
Let $2\le \lambda\le q-1$. Then our arguments leading to Theorem \ref{main} can also be used to deal with the generalized Hermitian Singer codes. 
\begin{thm}\label{main1} The generalized Hermitian Singer functional code 
$C_{\mathcal{L}}(\mathtt D,\lambda\mathtt G)$ is a $$[q(q^2-q+1), \textstyle{\frac{1}{2}}(2\lambda-1)(q^2-q)+\lambda+1, (q-\lambda)(q^2-q+1)]_{q^2}$$ linear code whose  minimum distance is equal to the designed minimum distance. 
\end{thm}
\begin{proof} Since $\deg(\lambda\mathtt{G})>q^2-q-2=2\mathfrak{g}-2$, the Riemann-Roch theorem yields that $\dim(\lambda\mathtt{G})=\lambda(q^2-q+1)-\mathfrak{g}+1=\frac{1}{2}(2\lambda-1)(q^2-q)+\lambda+1$. It remains to exhibit a function in $C_{\mathcal{L}}(\mathtt D,\lambda\mathtt G)$ which has as many as $\lambda(q^2-q+1)$ zeros in $D$.  According to Proposition \ref{pro02112025A}, take $\lambda$ Hermitian curve $H_t$ defined over $\Pi$  containing $G$ which are different from $H_\tau$. If $F_1(X,Y)=0, \ldots, F_\lambda(X,Y)=0$ are their equations, then the product $F_1\cdots F_\lambda$ is in $C_{\mathcal{L}}(\mathtt D,\lambda\mathtt G)$ and has exactly $\lambda(q^2-q+1)$ zeros in $D$.   
 \end{proof}

%\begin{rem}
%    Note that while $\lambda\geq q$ we get $n\leq \deg\mathtt G$ and hence the code falls in the class of \textit{weakly algebraic codes}, with less optimal properies.
%\end{rem}

\subsection{Hermitian Singer differential code}\label{sec5}
We keep upon our notation. Moreover, $C_\Omega(\mathtt D,\mathtt{G})$ stands for the Goppa (differential) code over the Hermitian curve $\mathcal{H}(q)$. We call $C_\Omega(\mathtt D,\mathtt{G})$ the \emph{Hermitian Singer differential} code. 

From \cite[Theorem 12.17]{P}, there is a canonical divisor $\mathtt{W}$ such that
$$C_\Omega(\mathtt D,{\mathtt{G}})\cong C_L(\mathtt D,{\mathtt{W}}+{\mathtt{D}}-{\mathtt{G}}).$$
The choice of $\mathtt{W}$ is not arbitrary, the requirement being the condition that for any point $P\in D$
\begin{equation}\label{eq31072025}
{\mathtt{W}}(P)+{\mathtt{D}}(P)-{\mathtt{G}}(P)=0.
\end{equation}
 Since ${\rm{div}}(dx)=(2\mathfrak{g}-2)Y_\infty$, it is  straightforward to verify that
$${\mathtt{W}}=\frac{F^2}{L} dx$$
satisfies \eqref{eq31072025} when ${\mathcal{H}}(q)$ is given in its canonical equation $y^q+y-x^{q+1}=0$, $\mathcal{C}$ is another Hermitian curve $F_q$ of equation $F(x,y)=0$ through the support of ${\mathtt{G}}$, and $L$ is the product of $q^2-q$ lines through an external point $R$ to ${\mathcal{H}}(q)$ together with the polar line of $R$ with respect to the unitary polarity associated with ${\mathcal{H}}(q)$.
We show that
$${\mathtt{W}}+{\mathtt{D}}-{\mathtt{G}}\equiv (q^3-q^2-5q-3)Y_\infty+ 2q{\mathtt{U}}$$
where ${\mathtt{U}}=U_1+U_2+U_3$ and $U_1,U_2,U_3$ are the common points of ${\mathcal{H}}(q)$ and $F$ in the cubic extension $PG(2,q^6)$ of $PG(2,q^2)$. As we have already observed in section \ref{sec3.1}, the intersection multiplicity $I(U_i,{\mathcal{H}}(q)\cap F)$ equals $q$ for $i=1,2,3$, and if $\ell_1\ell_2\ell_3$ is the cubic curve which is the product of the sides $\ell_i$ of the triangle $U_1U_2U_3$, then $I(U_{i},{\mathcal{H}}(q)\cap \ell_1\ell_2\ell_3)=q+1$ for $i=1,2,3$.
By Lemma \ref{lem13112025}, 
$$
\begin{array}{ccc}
{\mathtt{W}}+{\mathtt{D}}-{\mathtt{G}}\equiv \\
(2\mathfrak{g}-2)Y_\infty+2{\mathtt{G}}+ 2q{\mathtt{U}}+(q^3+1)Y_\infty-\\
2(q+1)^2Y_\infty
-\sum_{P\in {\mathcal{H}}(q)}P+{\mathtt{D}}-\mathtt{G}\equiv \\
(q^3-q^2-5q-3)Y_\infty+2q\mathtt{U}.
%(2\mathfrak{g}({\mathcal{H}}(q))-2)Y_{\infty}+2\mathtt{G}+2q \mathtt(U)\\
%-2(q+1)^2 Y_\infty-\sum_{P\in {\mathcal{H}}(q)}P+\sum_{P\in {\mathcal{H}}(q)}P-2\mathtt{G}
\\
\end{array}$$
Since $\mathtt{G}+q\mathtt{U}\equiv (q+1)^2Y_\infty$, this can also be written as
$$(q^2-1)(q+1)Y_\infty-2 \mathtt{G}.$$
Moreover, $$\deg({\mathtt{W}}+{\mathtt{D}}-{\mathtt{G}})=q^3-q^2+q-3.$$ Since $\deg(\mathtt W+\mathtt D-\mathtt G)>q^2-q-2=2\mathfrak{g}-2$, the Riemann-Roch theorem yields $$k=\deg(\mathtt W+\mathtt D-\mathtt G)-\mathfrak{g}+1=q^3-\textstyle{\frac{3}{2}}q^2+\frac{3}{2}q-2.$$
Also, the designed minimum distance is $\delta=(q^3+1)-(q^2-q+1)-\deg(\mathtt W+\mathtt D-\mathtt G)=3$.

We exhibit a function in $\mathcal{L}(\mathtt{W}+\mathtt{D}-\mathtt{G})=\mathcal{L}((q^2-1)(q+1)Y_\infty-2\mathtt{G})$
that has weight $q^2-q+1$. According to Propositions \ref{pro02112025A} and \ref{pro02112025B}, choose a Hermitian curve $H_t$ through $G$ other than $H_\tau$ together with $q-1$ more Hermitian curves $H_u$ such that they cover all points of $H_\tau\cap \Pi$ apart from a unique $\Gamma$-orbit $\Omega$. Let $F_t=F_t(X,Y)=0$, and $F_u=F_u(X,Y)=0$ be the equations of the chosen Hermitian curves. Set $F=\prod_u F_u$. Furthermore, choose  $\frac{1}{2}(q-2)(q+1)$ points in $\Omega$ distinct from $Y_\infty$, and let $Z=Z(X,Y)=0$ be the equation of a curve $\mathcal{C}$ of degree $q-2$ passing through the chosen points.  
 Let $\mathtt{R}=R_1+\ldots+ R_N$ with $N=q^3+1-2(q^2-q+1)$ and $R_i\in (H_\tau\cap\Pi)\setminus (\Omega\cup G)$ for $1\le i\le N$. Also, let $\mathtt{S}$ be the intersection divisor $H_\tau\circ C$. Now, define $f$ to be the function    
$$f=\frac{F_t^2FZ}{(\ell_1\ell_2\ell_3)^q}.$$
From \cite[Theorem 6.42]{HKT},  
$${\rm{div}}(f)=2\mathtt{G}+2q\mathtt{U}+\mathtt{R}+\mathtt{S}+(q-1)q\mathtt{U}+3q(q+1)Y_\infty-
\big((q+1)^3+(q-1)(q+1))Y_\infty+q(q+1)\mathtt{U}\big)$$
whence
$${\rm{div}}(f)=2\mathtt{G}-(q^2-1)(q+1)Y_\infty+\mathtt{R}+\mathtt{S}.$$ %$+2Y_\infty.$$ 
Therefore, $f\in \mathcal{L}({\mathtt{W}}+{\mathtt{D}}-{\mathtt{G}})$. 
Since the support of $\mathtt{S}$ contains at least $\frac{1}{2}(q-2)(q+1)$ points in $\Omega$ other than $Y_\infty$, $f$ has at least $q^3+1-2(q^2-q+1)+\frac{1}{2}(q-2)(q+1)$ zeros in $D$. Therefore, the weight of $f$ is at most $q^2-q+1-\frac{1}{2}(q-2)(q+1)=\frac{1}{2}(q^2-q)+2$. Thus, we proved the following result.
\begin{thm} 
\label{thm13112025}
    The Hermitian Singer differential code $C_{\Omega}(\mathtt D,\mathtt G)$ is a 
    $$[q(q^2-q+1),q^3-\frac{3}{2}q^2+\frac{3}{2}q-2,d]_{q^2}$$ linear code with $3\le d\leq \frac{1}{2}(q^2-q)+2$.
\end{thm}
\begin{rem}
\label{rem16112025} \rm{In the above construction of $f$, if the curve $\mathcal{C}$ happens to have at least $\frac{1}{2}(q-2)(q+1)+\kappa$ common points with $H(q)$ in $PG(2,q^2)$, then the bound on $d$ in Theorem \ref{thm13112025} is improved by $\kappa$ so that $3\le d \le \frac{1}{2}(q^2-q)+2-\kappa$. Unfortunately, finding such curves $\mathcal{C}$ appears to be a rather difficult problem, and no results about it are available in the literature yet.}
\end{rem}

\section{Computational results}\label{sec6}
In Section \ref{singergroup}, we have pointed out the usefulness of  a non-canonical representation of $PG(2,q^2)$ as a projective subplane $\Pi$ of $PG(2,q^6)$ when theoretical aspects of the Hermitian curve related with its Singer groups are investigated. However, this approach becomes a clear  disadvantage when computation, supported by software packages such as MAGMA \cite{magma}, is carried out over $\mathbb{F}_{q^6}$ instead of the much smaller $\mathbb{F}_{q^2}$. To work directly over $\mathbb{F}_{q^2}$, the Riemann-Roch space $\mathcal{L}(\mathtt{G})$ has to be redetermined on a Hermitian curve $H(q)$ in $PG(2,q^2)$ with respect to a Singer group $\Gamma$ of $PGU(3,q)$. This can be done by a change of the projective reference system. More precisely, starting off with $PG(2,q^2)$ equipped with a homogeneous coordinate system $(\bar{X}_1:\bar{X}_2:\bar{X}_0)$ over $\mathbb{F}_q$, regard $PG(2,q^6)$ as the cubic extension of $PG(2,q^2)$. Fix an element $a\in F_{q^3}\setminus \mathbb{F}_q$ such that $a^{q+1}+a+1=0$. Then $a^{q^2+1}+a^{q^2}+1=0$ and the determinant of the matrix 
\[M=\left(\begin{array}{ccc} 
a   & 1  & a^{q^2+1}\\
a^{q^2+1}  & a  & 1 \\
 1 & a^{q^2+1} & a
\end{array}\right)\]
does not vanish. Moreover, $M$ is orthogonal as $^tM=(1+a^2+a^{2(q^2+1)})M^{-1}$.  

The points   $A_1'(a:1:a^{q^2+1}),A_2'(a^{q^2+1}:a:1),A_0'(1:a^{q^2+1}:a)$ of $PG(2,q^6)$ are not collinear. Therefore, they may be taken together with the point $E=(1:1:1)$ as the vertices of a fundamental triangle with unity point $E$ of a projective coordinate system in $PG(2,q^6)$ with homogeneous coordinates $(X_1:X_2:X_0)$. It may be noticed that the Frobenius collineation fixes $E$ and that it preserves the triangle $A_1'A_2'A_0'$ and acts on it as the cycle $(A_1'A_0'A_2')$. Moreover, $^tM$ is the matrix that changes the coordinate system $(\bar{X}_1:\bar{X}_2:\bar{X}_0)$ to  $(X_1:X_2:X_0)$.  Therefore, a curve of homogeneous equation $U(X_1,X_2,X_0)=0$ has equation $\bar{U}(\bar{X}_1:\bar{X}_2:\bar{X}_0)$ where 
$$\bar{U}(\bar{X}_1:\bar{X}_2:\bar{X}_0)=U(a\bar{X}_1+\bar{X}_2+a^{q^2+1}\bar{X}_0,a^{q^2+1}\bar{X}_1+a\bar{X}_2+\bar{X}_0,\bar{X}_1+a^{q^2+1}\bar{X}_2+a\bar{X}_0).$$
\begin{rem}
    \label{rem16112025A}
\rm{The Hermitian curve $H_1$ of equation $X_1X_2^q+X_2X_0^q+X_0X_1^q=0$ in $PG(2,q^6)$ has equation $\bar{X}_1^{q+1}+\bar{X}_2^{q+1}+\bar{X}_0^{q+1}=0$ in $PG(2,q^2)$. Moreover, the generator $\omega$ given by the matrix $B$ in the coordinate system $(X_1:X_2:X_0)$ is represented by the symmetric matrix $^tMBM$ below in the coordinate system $(\bar{X}_1:\bar{X}_2:\bar{X}_0)$.} 
  \[\left(\begin{array}{ccc} 
\beta a^2+\beta^{q^2+1}a^{2(q^2+1)}+1 & a\beta+a^{q^2+2}\beta^{q^2+1}+a^{q^2+1}  & \beta a^{q^2+2}+a^{q^2+1}\beta^{q^2+1}+a\\
a\beta+a^{q^2+2}\beta^{q^2+1}+a^{q^2+1}  & \beta+\beta^{q^2+1}a^2+a^{2(q^2+1)}  &  \beta a^{q^2+1}+\beta^{q^2+1}a+a^{q^2+2}\\
\beta a^{q^2+2}+a^{q^2+1}\beta^{q^2+1}+a & \beta a^{q^2+1}+\beta^{q^2+1}a+a^{q^2+2} & a^{2(q^2+1)}\beta+\beta^{q^2+1}+a^2
\end{array}\right)\]
\end{rem}
\subsection{Case q=3} For this case, a Magma aided computation regarding Theorem \ref{thm13112025} shows that $d=5$, and hence the upper bound on $d$ in Theorem \ref{the08112025} is strict. 
\subsection{Case q=4} In our computation in $PG(2,16)$, we have used Remark \ref{rem16112025A}. Therefore, in this case, the Hermitian curve $H$ is taken with homogeneous equation   $X_1^{q+1}+X_2^{q+1}+X_0^{q+1}=0$, and 
the Singer group $\Gamma$ is generated by $^tMBM$ up to a constant factor, namely   
\[\left(\begin{array}{ccc} 
r^6  & r^{2}  & 1\\
r^{2}  & r^{14}  & r^{8} \\
 1 & r^{8} & r^3
\end{array}\right)\]
Then the $\Gamma$-orbit $\Omega$ of the point $P_0(0:1:r^3)$ consists of the following $13$ points of $H$:  
\begin{align*}
&P_0(0:1:r^3), P_1(1:r^4:r^8), P_2(1:r^3:0), P_3(1:0:r^3),
P_4(1:r^7:r^{11}),\\ 
&P_5(1:r^4:r^{11}), P_6(1:r:r^{14}), 
P_7(1:r^{13}:r^{14}), P_8(1:0:r^{12}), P_{9}(1:r^{12}:0),\\
&P_{10}(1:r:r^2), P_{11}(0:1:r^{12}), P_{12}(1:r^5:r^{10}). 
\end{align*}
A Magma aided exhaustive computation finds thirteen (irreducible) conics containing at least seven points in $\Omega$, but none of them contains eight points
from $\Omega$. One of these conics has homogeneous equation $X_1^2+r^3X_1X_2+X_1X_0+r^{11}X_2X_0+r^8 X_0^2$ and passes through the following seven points of $\Omega$: $P_0,P_1,P_3,P_4,P_9,P_{10},P_{12}$.  Therefore, the bound in Theorem \ref{the08112025} for $q=4$ can be refined into 
$\frac{1}{2}(16-4)+2-2=6$. Finally, a Magma aided computation regarding Theorem \ref{thm13112025} shows that $d=6$, and hence the refined upper bound on $d$ in Theorem \ref{the08112025} is strict.

\section*{Acknowledgements} G. Korchm\'aros, F. Romaniello  and V. Smaldore   have been partially supported by the Italian National Group for Algebraic and Geometric Structures and their Applications (GNSAGA - INdAM).

\end{document}